\documentclass[12pt]{amsart}
\usepackage{amscd,amsmath,amsthm,amssymb}
\usepackage[left]{lineno}
\usepackage{pst-plot,pst-3d}
\usepackage[T1]{fontenc}
\usepackage{color}
\usepackage{pstricks}
\usepackage{stmaryrd}
\usepackage[utf8]{inputenc}
\usepackage{pstricks-add}
\usepackage{graphicx}
\usepackage{epstopdf}
\usepackage{graphicx}
\usepackage{cleveref}
\usepackage{leftidx}
\usepackage{tikz}
\usepackage{tikz-cd}
\usepackage{url}
\usepackage{booktabs} 
\usepackage{colortbl}
\usepackage{enumitem}
\definecolor{verylight}{gray}{0.97}
\definecolor{light}{gray}{0.9}
\definecolor{medium}{gray}{0.85}
\definecolor{dark}{gray}{0.6}

 \def\NZQ{\mathbb}               % the font for N,Z,Q,R,C
 \def\NN{{\NZQ N}}
 
 \def\ZZ{{\NZQ Z}}

 \def\FF{{\NZQ F}}

 \def\G{{\mathcal G}}

 \def\A{{\mathcal A}}

 \def\slex{\textup{slex}}
 \def\Tor{\textup{Tor}}
 \def\indeg{\textup{indeg}}
 \def\lcm{\textup{lcm}}

 \def\0b{{\mathbf 0}}

 \def\opn#1#2{\def#1{\operatorname{#2}}} % to make operators
 \opn\chara{char} \opn\length{\ell} \opn\pd{pd} \opn\rk{rk}
 \opn\projdim{proj\,dim} \opn\injdim{inj\,dim} \opn\rank{rank}
 \opn\depth{depth} \opn\grade{grade} \opn\height{height}
 \opn\embdim{emb\,dim} \opn\codim{codim}
 
 \opn\Tr{Tr} \opn\bigrank{big\,rank}
 \opn\superheight{superheight}\opn\lcm{lcm}
 \opn\trdeg{tr\,deg}%\emph{
 \opn\reg{reg} \opn\lreg{lreg} \opn\ini{in} \opn\lpd{lpd}
 \opn\size{size} \opn\sdepth{sdepth}
 \opn\link{link}\opn\fdepth{fdepth}\opn\lex{lex}
 \opn\tr{tr}
 \opn\type{type}
 \opn\gap{gap}
 \opn\arithdeg{arith-deg}
 \opn\HS{HS}
 \opn\tet{tet}
 \opn\div{div} \opn\Div{Div} \opn\cl{cl} \opn\Cl{Cl}
 \opn\Spec{Spec} \opn\Supp{Supp} \opn\supp{supp} \opn\Sing{Sing}
 \opn\Ass{Ass} \opn\Min{Min}\opn\Mon{Mon}
 \opn\Ann{Ann} \opn\Rad{Rad} \opn\Soc{Soc}\opn\Deg{Deg} \opn\Gen{Gen}\opn{\Im}{Im}
 \opn\Im{Im} \opn\Ker{Ker} \opn\Coker{Coker} \opn\Am{Am}
 \opn\Hom{Hom} \opn\Tor{Tor} \opn\Ext{Ext} \opn\End{End}
 \opn\Aut{Aut} \opn\id{id}
 
 \opn\nat{nat}
 \opn\pff{pf}%   \pf exists already
 \opn\Pf{Pf} \opn\GL{GL} \opn\SL{SL} \opn\mod{mod} \opn\ord{ord}
 \opn\Gin{Gin} \opn\Hilb{Hilb}\opn\sort{sort}
 \opn\PF{PF}\opn\Ap{Ap}
 \opn\mult{mult}
 \opn\bight{bight}
  \opn\bg{bg}
   \opn\gcl{gcl}
 \opn\aff{aff}
 \opn\relint{relint} \opn\st{st}
 \opn\lk{lk} \opn\cn{cn} \opn\core{core} \opn\vol{vol}  \opn\inp{inp} \opn\nilpot{nilpot}
 \opn\link{link} \opn\star{star}\opn\lex{lex}\opn\set{set}
 \opn\width{wd}
 \opn\Fr{F}
 \opn\QF{QF}
 \opn\G{G}
 \opn\type{type}\opn\res{res}
 \opn\conv{conv}
 \opn\Ind{Ind}
 \opn\soc{soc}
 \opn\gr{gr}
 
 \def\pot#1#2{#1[\kern-0.28ex[#2]\kern-0.28ex]}

 \opn\dirlim{\underrightarrow{\lim}}
 \opn\inivlim{\underleftarrow{\lim}}
 \def\Implies{\ifmmode\Longrightarrow \else
         \unskip${}\Longrightarrow{}$\ignorespaces\fi}
 \def\implies{\ifmmode\Rightarrow \else
         \unskip${}\Rightarrow{}$\ignorespaces\fi}
 \def\iff{\ifmmode\Longleftrightarrow \else
         \unskip${}\Longleftrightarrow{}$\ignorespaces\fi}

 \let\:=\colon
 \newtheorem{Theorem}{Theorem}[section]
 
 \newtheorem{Corollary}[Theorem]{Corollary}
 \newtheorem{Proposition}[Theorem]{Proposition}
 \theoremstyle{definition}
 \newtheorem{Remark}[Theorem]{Remark}
 
 \newtheorem{Example}[Theorem]{Example}
 
 \newtheorem{Definition}[Theorem]{Definition}
 
 \newtheorem{Crit}[Theorem]{Criterion}

 \newtheorem{Question}[Theorem]{Question}
 \let\epsilon\varepsilon
 \let\kappa=\varkappa
 \opn\dis{dis}
 \def\pnt{{\raise0.5mm\hbox{\large\bf.}}}
 
 \opn\Lex{Lex}

\begin{document}
 \title{Linear resolutions of $t$--spread lexsegment ideals \\ via Betti splittings}
 \author{Marilena Crupi, Antonino Ficarra}
 
 \address{Marilena Crupi, Department of Mathematics and Computer Sciences, Physics and Earth Sciences, University of Messina, Viale Ferdinando Stagno d'Alcontres 31, 98166 Messina, Italy} \email{mcrupi@unime.it}
 
 \address{Antonino Ficarra, Department of Mathematics and Computer Sciences, Physics and Earth Sciences, University of Messina, Viale Ferdinando Stagno d'Alcontres 31, 98166 Messina, Italy} \email{antficarra@unime.it}
 
 \date{}
 \maketitle
 \begin{abstract} Let $S=K[x_1,\dots,x_n]$ be a polynomial ring in $n$ variables with coefficients over a field $K$. A $t$--spread lexsegment ideal $I$ of $S$ is a monomial ideal generated by a $t$--spread lexsegment set. We determine all $t$--spread lexsegment ideals with linear resolution by means of Betti splittings. As applications we provide formulas for the Betti numbers of such a class of ideals and furthermore we characterize all incompletely $t$--spread lexsegment ideals with linear quotients.
 \end{abstract}
 
 \section*{Introduction}
 The classification of 
 ideals with linear resolution is of great importance in commutative algebra. In general, proving
 that a graded ideal of a standard graded polynomial ring has linear resolution can be difficult. However,
 some classes of graded ideals with linear resolution may be found in \cite{AHH2, ADH, EOS, EO, DH, FC1, HMRZ, Lu2022, YP} and the reference therein. 
 Let $S=K[x_1,\dots,x_n]$ be the polynomial ring in $n$ variables with coefficients over a fixed field $K$. One method for computing Betti numbers of a monomial ideal $I$ of $S$ is to decompose $I$ into smaller monomial ideals $P$ and $Q$ such that $I = P + Q$ and the minimal set of monomial generators $G(I)$ of $I$ is the disjoint union of the minimal sets of monomial generators of $P$ and $Q$. By taking this approach in \cite{EK}, Eliahou and Kervaire introduced the notion of \emph{splittable monomial ideal} and proved that for such an ideal $I$ there exists a formula for computing its Betti numbers in terms of the Betti numbers of two suitable proper non--zero ideals $P, Q$ contained in $I$, and the Betti numbers of $P\cap Q$. More in detail, $\beta_q(I)=\beta_q(P)+\beta_q(Q)+\beta_{q-1}(P\cap Q)$, for all $q$. Verifying that an ideal $I$ is splittable is not very simple, however, there are many ways of manipulating the minimum set of monomial generators of a monomial ideal $I$ to construct smaller ideals $P$ and $Q$   such that the formula above still holds. Motivated overall by this fact afterwards Francisco, Ha and Van Tuyl have introduced the concept of \emph{Betti splitting} \cite{FHT2009}: let $I$, $P$ and $Q$ be monomial ideals of $S$ such that $G(I)$ is the disjoint union of $G(P)$ and $G(Q)$. The ideal $I=P+Q$ is said to be a \textit{Betti splitting} if
 $\beta_{i,j}(I)=\beta_{i,j}(P)+\beta_{i,j}(Q)+\beta_{i-1,j}(P\cap Q)$, for all $i, j$.  This strategy implies that one knows some information about the minimal
 resolutions of $P$, $Q$ and $P\cap Q$, and thus it is natural to analyze conditions on the Betti
 numbers of those ideals that force $I = P+Q$ to be a Betti splitting. This notion will be crucial for the development of the article. Furthermore, if $I = P+Q$ is a Betti splitting and $P\cap Q=PQ$, to compute the (graded) Betti numbers of $I$, one can take also into account the formula for the (graded) Betti numbers of the product of two ideals, see \cite[Proposition 3.1]{HMRZ}.
 
 In \cite{CAC}, the authors introduced the notion of $t$--spread lexsegment ideals.  For $t\ge 1$, this new class of monomial ideals generalizes the notion of squarefree lexsegment ideal \cite{AHH2}. A squarefree monomial ideal $I$ of $S$ is a squarefree lexsegment ideal if for all squarefree monomials $u\in I$ and all squarefree monomials $v\in S$ with $\deg(u)=\deg(v)$ and such that $v\ge_{\slex}u$, then $v\in I$. Such a notion can be generalized as follows \cite{CAC}: let $t\ge 1$. A $t$--spread ideal $I$ of $S$ is said to be a \emph{$t$--spread lexsegment ideal} if for all $t$--spread monomials $u\in I$ and all $t$--spread monomials $v\in S$ with $\deg(u)=\deg(v)$ and such that $v\ge_{\slex}u$, then $v\in I$. If $t\ge 0$ is an integer,  a monomial $u=x_{i_1}x_{i_2}\cdots x_{i_d}$ of $S$ is \textit{$t$--spread} if $i_{j+1}-i_j\ge t$, for all $j=1,\dots,d-1$ and a monomial ideal of $S$ is $t$--spread if it is generated by $t$--spread monomials \cite{EHQ}. 
 
 In \cite{HM}, a general lexsegment ideal was introduced. Such a definition has been recently generalized in \cite{FC1} by the notion of an arbitrary $t$--spread lexsegment ideal. 
 Let $t\ge 1$ and let $M_{n,d,t}$ be the set of all $t$--spread monomials of $S$ of degree $d$. Assume that $M_{n,d,t}$ is endowed with the squarefree lexicographic order $\ge_{\slex}$. Let $u,v\in M_{n,d,t}$ with $u\ge_{\slex}v$, then the set $\mathcal{L}_t(u,v)=\{w\in M_{n,d,t}:u\ge_{\slex}w\ge_{\slex}v\}$ is called  \emph{$t$--spread lexsegment set}.
 If $u$ is the maximum monomial of $M_{n,d,t}$, or $v$ is the minimum monomial of $M_{n,d,t}$, $\mathcal{L}_t(u,v)$ is said to be an \emph{initial} or a \emph{final $t$--spread lexsegment set}, respectively, and the ideal $(\mathcal{L}_t(u,v))$ is said to be an \emph{initial} or a \textit{final $t$--spread lexsegment ideal}. One can observe that the \emph{initial $t$--spread lexsegment ideal}, defined above, is the $t$--spread lexsegment ideal introduced in \cite{CAC}. Thus, in this article, we will refer to such an ideal as the \emph{initial $t$--spread lexsegment ideal}, and
 we call \emph{$t$--spread lexsegment ideal} an ideal generated by a $t$--spread lexsegment $\mathcal{L}_t(u,v)$. For $t=0$, one obtains the definition in \cite{HM}. A $t$--spread lexsegment ideal $I = (\mathcal{L}_t(u,v))$ is said to be a \emph{completely $t$--spread lexsegment ideal} if $I$ is the intersection of an initial and a final $t$--spread lexsegment ideal. Otherwise, $I$ is said to be an \emph{incompletely $t$--spread lexsegment ideal}.
 
 The aim of this article is to classify all $t$--spread lexsegment ideals with a linear resolution \emph{via} Betti splittings, where $t$ is a positive integer. The classification is achieved by analyzing the behaviour of the incompletely $t$--spread lexsegment ideals. Such a study began in \cite{FC1} using \emph{classical} techniques and all completely $t$--spread lexsegment ideals with a linear resolution were characterized.  The case $t=0$ were considered in \cite{ADH, DH}, whereas the case $t=1$ was analyzed in \cite{BS}. On one point of view, the article provides a further contribution to the more general problem of determining all monomial ideals with linear resolutions.
 
 The article is organized as follows. Section \ref{sec1} contains some notions and preliminary results that will be used in the article. We particularly focus on the notion of Betti splitting (Subsection \ref{subsection2}) and its connection with linear resolutions (Proposition \ref{prop:CharactI=P+QLinear}).  Section \ref{sec2} analyzes incompletely $t$--spread lexsegment ideals with linear resolution. Firstly, we observe the existence of a natural Betti splitting for any $t$--spread lexsegment ideal (Corollary \ref{Cor:IlinResNotCompletely}). This result together with Proposition \ref{prop:CharactI=P+QLinear} allows us to obtain all the ingredients for getting the main result in the article which characterizes all incompletely $t$--spread lexsegment ideals with a linear resolution (Theorem \ref{teor:InotCompLinRes}). Section \ref{sec3} describes some applications of the results in the previous sections. We state a formula that gives the Betti numbers of an incompletely $t$--spread lexsegment ideal $I$ of $S$ with a linear resolution (Theorem \ref{thm:Bettiformula}). The case when $I$ is a completely $t$--spread lexsegment ideal was faced and solved in \cite{FC1}. Furthermore, we characterize all incompletely $t$--spread lexsegment ideals with linear quotients, for $t\ge 1$. The case $t=0$ was considered in \cite{EOS}. Finally, Section 4 contains our conclusions and perspectives.
 
 \section{Preliminaries}\label{sec1}
 Let $K$ be a field and $S=K[x_1, \ldots, x_n]$ be the standard graded polynomial ring with coefficients in $K$. For a graded ideal $I$ of $S$, we denote by $\beta_{i,j}(I) = \dim_K \Tor_i^S(I, k)_j$ the graded Betti numbers of $I$ and by $\beta_i(I) = \sum_{j\in \ZZ}\beta_{i,j}(I)$ the Betti numbers of $I$. Given a monomial ideal $I$ of $S$, let $G(I)$ be the minimal system of monomial generators of $I$ and denote by $G(I)_d$ the set of all $u\in G(I)$ with $\deg(u)=d$, $d>0$. The lowest degree of a monomial of $G(I)$ is denoted by $\indeg(I)$. 
 
 If $I$ is generated in degree $d$, we say that $I$ has a $d$--linear resolution if $\beta_{i, i+j}(I)=0$ for every $i, j\in \NN$ and $j\neq d$. Usually, when the context is clear, we simply write that $I$ has a linear resolution. 
 
 The \emph{Castelnuovo--Mumford regularity} of a graded ideal $I$ is defined as follows:
 \[\reg(I) = \max\{j-i: \beta_{i, j}(I)\neq 0\}.\]
 A graded ideal $I$ has a $d$--linear resolution if and only if $\reg(I)=d$ and $I$ is generated in degree $d$. Moreover, if $I$ is a graded ideal of $S$ with homogeneous generators $m_1, \ldots, m_r$, then $\reg(I)\ge \deg(m_i)$, for all $i$ and, in particular, $\reg(I)\ge \indeg(I)$.
 
 Every monomial $w\in S$ can be written as $w=x_{i_1}x_{i_2}\cdots x_{i_d}$ with sorted indices $1\le i_1\le i_2\le\dots\le i_d\le n$. We define the \textit{support} of $w$ as follows:
 $$
 \supp(w)=\{j: x_j \,\mbox{divides $w$} \}= \{i_1,i_2,\dots,i_d\}.
 $$
 With the notation above,  $\min(w)=\min\big\{i:i\in\supp(u)\big\}=i_1$, $\max(w)=\max\big\{i:i\in\supp(u)\big\}= i_d$. Moreover, we set $\max(1)=\min(1)=0$.
 If $i_1< i_2< \cdots < i_d$, the monomial $w$ is said to be \textit{squarefree}.
 
 \subsection{$t$--spread lexsegment ideals}\label{subsection1}
 
 This subsection is an overview on $t$--spread lexsegment ideals. We refer to \cite{CAC, EHQ, FC1} for more details on the subject.
 \begin{Definition}
 	\rm Given $n\ge1,t\ge0$ and $u=x_{i_1}x_{i_2}\cdots x_{i_d}\in S$, with $1\le i_1\le i_2\le\ldots\le i_d\le n$, we say that $u$ is \textit{$t$--spread} if $i_{j+1}-i_j\ge t$, for all $j=1,\dots,d-1$. We say that a monomial ideal $I$ of $S$ is \textit{$t$--spread} if it is generated by $t$--spread monomials.
 \end{Definition}
 
 We note that any monomial ideal of $S$ is a $0$--spread monomial ideal and any squarefree monomial ideal of $S$ is a $1$--spread monomial ideal.
 
 Let $n,d,t\ge 1$.  We denote by $M_{n,d,t}$ the set of all $t$--spread monomials of $S$ of degree $d$. Throughout the article we tacitly assume that $n\ge 1+(d-1)t$, otherwise $M_{n,d,t}=\emptyset$. Furthermore, we assume that $M_{n,d,t}$ is endowed with the \emph{squarefree lexicographic order}, $\ge_{\slex}$ \cite{AHH2}. Recall that given $u=x_{i_1}x_{i_2}\cdots x_{i_d}, v=x_{j_1}x_{j_2}\cdots x_{j_d} \in M_{n,d,t}$, with $1\le i_1<i_2<\dots< i_d\le n$, $1\le j_1<j_2<\dots<j_d\le n$, then $u>_{\slex}v$ if
 $i_1=j_1,\ \dots,\ i_{s-1}=j_{s-1}\ \ \text{and}\ \ i_s<j_s$,
 for some $1\le s\le d$. \\
 Let $T$ be a not empty subset of $M_{n,d,t}$. We denote by $\max(T)$ ($\min(T)$, respectively) the maximum (minimum, respectively) monomial $w\in T$ with respect to $\ge_{\slex}$. It is $\max(M_{n,d,t}) = x_1x_{1+t}x_{1+2t}\cdots x_{1+(d-1)t}$ and $\min(M_{n,d,t}) = x_{n-(d-1)t}x_{n-(d-2)t}\cdots x_{n-t}x_n$.
 
 In \cite{FC1}, the following definitions have been introduced.
 \begin{Definition}\rm
 	Let $u,v\in M_{n,d,t}$, $u\ge_{\slex}v$.
 	The set
 	\[\mathcal{L}_t(u,v) = \big\{w\in M_{n,d,t}:u\ge_{\slex}w\ge_{\slex}v\big\}\]
 	is called \textit{$t$--spread lexsegment set}. The set
 	\[\mathcal{L}_t^{i}(v) = \big\{w\in M_{n,d,t}:w\ge_{\slex}v\big\}=\mathcal{L}_t(\max(M_{n,d,t}),v)\]
 	is said to be an \emph{initial $t$--spread lexsegment set} and the set
 	\[\mathcal{L}_t^{f}(u)=\big\{w\in M_{n,d,t}:w\le_{\slex}u\big\}=\mathcal{L}_t(u,\min(M_{n,d,t}))
 	\]
 	is said to be a \emph{final $t$--spread lexsegment set}.
 \end{Definition}
 
 It is clear that for $u,v\in M_{n,d,t}$ with $u\ge_{\slex}v$, then $\mathcal{L}_t(u,v)=\mathcal{L}_t^i(v)\cap\mathcal{L}_t^f(u)$.
 
 \begin{Definition}\label{def:comp}\rm
 	A $t$--spread monomial ideal $I$ of $S$ is said to be a \textit{$t$--spread lexsegment ideal} if it is generated by a $t$--spread lexsegment. \\A $t$--spread lexsegment ideal $I=(\mathcal{L}_t(u,v))$ is a \textit{completely $t$--spread lexsegment ideal} if
 	$
 	I=J\cap T,
 	$
 	where $J=(\mathcal{L}_t^i(v))$ and $T=(\mathcal{L}_t^f(u))$. Otherwise, $I$ is said to be an \textit{incompletely $t$--spread lexsegment ideal}.
 \end{Definition}
 
 One can easily observe that every initial and every final $t$--spread lexsegment ideal is completely $t$--spread lexsegment, but not all $t$--spread lexsegment ideals are completely $t$--spread lexsegment ideals (see, for instance, \cite[Example 2.5]{FC1}). 
 In \cite{FC1}, all completely $t$--spread lexsegment ideals were characterized.
 
 Now, we recall some notions that will be useful in the sequel.
 
 A $t$--spread monomial ideal $I$ is said to be \textit{$t$--spread strongly stable} if for all $t$--spread monomial $u\in I$, all $j\in \supp(u)$ and all $i< j$ such that $x_i(u/x_j)$ is $t$--spread, it follows that $x_i(u/x_j)\in I$.

 Since every initial $t$--spread lexsegment ideal $I$ of $S$ is a $t$--spread strongly stable ideal, in order to compute the graded Betti numbers of $I$ one can use the Ene, Herzog, Qureshi's formula (\cite[Corollary 1.12]{EHQ}):
 \begin{equation}
 \label{eq1}
 \beta_{i,i+j}(I)\ =\ \sum_{u\in G(I)_j}\binom{\max(u)-t(j-1)-1}{i}.
 \end{equation}
 
 A $t$--spread final lexsegment ideal $I$ of $S$ is also a $t$--spread strongly stable ideal, but with the order of the variables reversed ($x_n> x_{n-1} > \cdots >x_1$). Thus to compute the graded Betti numbers of $I$ one can use the following \emph{modified} Ene, Herzog, Qureshi's formula \cite{FC1}:
 \begin{equation}\label{eq2}
 \beta_{i,i+j}(I)\ =\ \sum_{u\in G(I)_j}\binom{n-\min(u)-t(j-1)}{i}.
 \end{equation}
 
 Let $>_{\lex}$ be the usual \textit{lexicographic order} on $S$ with $x_1>x_2>\cdots>x_n$ \cite{JT}. The next theorem collects some results from  \cite{FC1} that will be fundamental for the aim of the article.
 \begin{Theorem}\textup{\cite{FC1}}.\label{thm:compltspreadlex}
 	Given $n,d,t\ge1$, let $u=x_{i_1}x_{i_2}\cdots x_{i_d}$ and $v=x_{j_1}x_{j_2}\cdots x_{j_d}$ be $t$--spread monomials of degree $d$ of $S$ such that $u\ge_{\slex}v$. Let $I=(\mathcal{L}_t(u,v))$.
 	\begin{enumerate}\item[\em(i)] The following conditions are equivalent:
 		\begin{enumerate}
 			\item[\em(a)] $I$ is a completely $t$--spread lexsegment ideal;
 			\item[\em(b)] for every $w\in M_{n,d,t}$ with $w<_{\slex}v$ there exists an integer $s>i_1$ such that $x_s$ divides $w$ and $x_{i_1}(w/x_{s})\le_{\lex}u$. 
 		\end{enumerate}
 		\item[\em(ii)] Suppose $I$ is a completely $t$--spread lexsegment ideal with $\min(v)>\min(u)=1$. $I$ has a linear resolution if and only if one of the following conditions holds:
 		\begin{enumerate}
 			\item[\em(a)] $i_2=1+t$;
 			\item[\em(b)] $i_2>1+t$ and for the largest $w\in M_{n,d,t}$, $w<_{\slex}v$, we have $x_1(w/x_{\max(w)})$ $\le_{\lex}x_1x_{i_2-t}x_{i_3-t}\cdots x_{i_d-t}$.
 		\end{enumerate}
 		\item[\em(iii)] Suppose $I$ is a completely $t$--spread lexsegment ideal with a linear resolution. Then, for all $i\ge0$,
 		$$
 		\beta_{i}(I)=\sum_{w\in\mathcal{L}_t^f(u)}\binom{n-\min(w)-(d-1)t}{i}-\sum_{\substack{w\in\mathcal{L}_t^f(v)\\ w\ne v}}\binom{\max(w)-(d-1)t-1}{i}.
 		$$
 	\end{enumerate}
 \end{Theorem}
 
 \subsection{Betti numbers and Betti splittings}\label{subsection2}
 
 In this subsection we discuss the notions of splittable monomial ideal \cite{EK} and Betti splitting \cite{FHT2009}.
 In \cite{EK}, Eliahou and Kervaire introduced the notion of \emph{splittable monomial ideal}. 
 
 \begin{Definition}\rm \label{def:splittable}A monomial ideal $I$ of $S$ is \emph{splittable} if $I$ is the sum of two non--zero monomial ideals $P$ and $Q$ of $S$ such that $G(I)$ is the disjoint union of $G(P)$ and $G(Q)$ and there is a \emph{splitting function}
 	\begin{align*}
 	G(P\cap Q)\ &\longrightarrow\ G(P)\times G(Q)\\
 	w\qquad \ &\longmapsto\ (\phi(w), \psi(w))
 	\end{align*}
 	satisfying the following properties:
 	\begin{enumerate}
 		\item[(S1)] $w = \lcm(\phi(w), \psi(w))$, for all $w\in G(P\cap Q)$,
 		\item[(S2)] for every subset $G'$ of $G(P\cap Q)$, both $\lcm(\phi(G'))$ and $\lcm(\psi(G'))$ strictly divide $\lcm(G')$.
 	\end{enumerate}
 \end{Definition}
 
 The pair $P, Q$ satisfying the above conditions is called a \emph{splitting of $I$}.
 Moreover, in \cite[Proposition 3.1]{EK} (see, also, \cite[Theorem 3.2]{HT2007}) it was proved that if $I$ is a splittable ideal with splitting $P, Q$, then for all $i, j\ge 0$, 
 \begin{equation}\label{eq1:BettiSplittingI=P+Q}
 \beta_{i,j}(I)=\beta_{i,j}(P)+\beta_{i,j}(Q)+\beta_{i-1,j}(P\cap Q),
 \end{equation}
 where $\beta_{-1,0}=0$. Hence, for all $q\ge 0$,
 \begin{equation}\label{eq:totalBettiSplittingI=P+Q}
 \beta_q(I)=\beta_q(P)+\beta_q(Q)+\beta_{q-1}(P\cap Q),
 \end{equation}
 with $\beta_{-1}=0$.
 
 From now on, if $I$ is a monomial ideal which is splittable in the sense of Eliahou and Kervaire, we will say that $I$ is \emph{Eliahou--Kervaire splittable} (E--K splittable, for short) and if $P, Q$ is a \emph{splitting of $I$}, the ideal $I=P+Q$ will be called \emph{Eliahou--Kervaire splitting} (E--K splitting, for short).
 In \cite{EK}, it was proved that every stable ideal is E--K splittable. Now, we want to verify that a $t$--spread strongly stable ideal $I$ is E--K splittable.
 
 Let $I$ be a $t$--spread strongly stable ideal  of $S=K[x_1, \ldots, x_n]$. Let $P$ be the ideal of $S$ whose minimal set of monomial generators consists of all $u\in G(I)$ such that $x_1$ divides $u$. It is clear that $P$ is a $t$--spread strongly stable ideal of $S$.
 Consider the monomial ideal $Q$ such that $G(Q)=G(I)\setminus G(P)$. One can quickly verify that $Q$ is a $t$--spread strongly stable ideal of $K[x_2, \ldots, x_n]$.\\
 Next, we prove that $W=P\cap Q$ is equal to $x_1Q$. The inclusion $W\subseteq x_1Q$ is clear. Let $v \in G(Q)$, then $x_1v\in I$. Consider the monomial $\tilde v =x_1v/x_{\min(v)}$. 
 Since $v$ is a $t$-spread monomial and $1<\min(v)$, it yields that $\tilde v$ is $t$--spread and $\tilde v\in I$ because $I$ is $t$--spread strongly stable. Hence, from \cite[Lemma 1.3]{EHQ}, $\tilde v=uy$, with $u\in G(I)$, $\max(u)<\min(y)$. It follows that $x_1$ divides $u$ and $u\in G(P)$. Thus $x_1v\in P\cap Q$ and the claim follows. Note that $G(P\cap Q) = x_1G(Q)$.
 
 Define the following map
 \begin{align*}
 G(P\cap Q)\ &\longrightarrow\ G(P)\times G(Q)\\
 w\ \qquad &\longmapsto\ (\phi(w)=u, \psi(w)=w/x_1),
 \end{align*}
 where $u$ is determined as follows: write $w= x_1v$, $v\in G(Q)$ and consider the canonical decomposition of the $t$--spread monomial $w/x_{\min(v)} = x_1v/x_{\min(v)}$, \emph{i.e.}, $w/x_{\min(v)}=uy$, with $u\in G(I)$, $\max(u)<\min(y)$ \cite[Lemma 1.3]{EHQ}. It is easy to verify that the above map is a splitting function of $I$ and thus $I=P+Q$ is an E--K splitting. On the other hand, since the ideal $Q$ in $S$ has the same (graded) Betti numbers of $Q$ in $K[x_2, \ldots, x_n]$ and $P\cap Q=x_1Q$, it follows that  
 \[
 \beta_{i,j}(I)=\beta_{i,j}(P)+\beta_{i,j}(Q)+\beta_{i-1,j}(P\cap Q)= \beta_{i,j}(P)+\beta_{i,j}(Q)+\beta_{i-1,j-1}(Q),
 \]
 for all $i, j$, and 
 \[\beta_q(I)=\beta_q(P)+\beta_q(Q)+\beta_{q-1}(Q),\,\, \textrm{for all $q$}.\] 
 Finally, in order to compute the (graded) Betti numbers of the $t$--spread strongly stable ideal  $I = P+Q$, we need only to compute the (graded) Betti numbers of $P$ and $Q$.

 The notion of splittable ideal has been used in many contexts, see \cite{FHT2009} and the reference therein. Nevertheless, the construction of the required splitting function or even to know whether such a function does exist is not easy. In \cite{FHT2009}, Francisco, Ha and Van Tuyl pointed out that there are other conditions on $P$ and $Q$ beyond the criterion of Eliahou and Kervaire that imply that formula (\ref{eq1:BettiSplittingI=P+Q}) holds.  Such situations motivated them to introduce the next definition.
 
 \begin{Definition}\rm\label{def:Bettisplitting}
 	Let $I$, $P$, $Q$ be monomial ideals of $S$ such that $G(I)$ is the disjoint union of $G(P)$ and $G(Q)$. We say that $I=P+Q$ is a \textit{Betti splitting} if
 	\begin{equation}\label{eq:BettiSplittingI=P+Q}
 	\beta_{i,j}(I)=\beta_{i,j}(P)+\beta_{i,j}(Q)+\beta_{i-1,j}(P\cap Q), \ \ \ \textup{for all}\ i,j.
 	\end{equation}
 \end{Definition}
 
 From what we have verified previously, it follows that for every $t$--spread strongly stable ideal $I$ of $S$ there exist two monomial ideals $P$, $Q$ of $S$ such that $I=P+Q$ is a \textit{Betti splitting}. 
 \begin{Example} \rm Let $$I = (x_1x_3, x_1x_4, x_1x_5, x_1x_6x_8, x_1x_6x_9, x_1x_7x_9, x_2x_4x_6x_8,x_2x_4x_6x_9,x_2x_4x_7x_9)$$
 	a $2$--spread strongly stable ideal of $S=K[x_1, \ldots, x_9]$. Setting
 	$P= (x_1x_3, x_1x_4, x_1x_5,$ $ x_1x_6x_8, x_1x_6x_9, x_1x_7x_9)$ and $Q=(x_2x_4x_6x_8,x_2x_4x_6x_9,x_2x_4x_7x_9),$ one can quickly observe that $I = P+Q$ and that $G(P)\cap G(Q)=\emptyset$, and furthermore $P\cap Q=(x_1x_2x_4x_6x_8,x_1x_2x_4x_6x_9,x_1x_2x_4x_7x_9)=x_1Q$.
 	Using \textit{Macaulay2} \cite{GDS}, the Betti tables of $I,$ $P$, $Q$, $P\cap Q$ are the following ones:
 	\begin{align*}
 	\arraycolsep=4.8pt\def\arraystretch{0.8}
 	\begin{array}{ccccccccc}
 	\begin{matrix}
 	&0&1&2&3&4\\\text{total:}&9&19&18&9&2\\
 	\text{2:}& 3&3&1&\text{.}&\text{.}\\
 	\text{3:}&3&11&15&9&2\\
 	\text{4:}&3&5&2&\text{.}&\text{.}\\	
 	\end{matrix}&&&&&&&&
 	\begin{matrix}
 	&0&1&2&3&4\\
 	\text{total:}&6&14&16&9&2\\
 	\text{2:}&3&3&1&\text{.}&\text{.}\\
 	\text{3:}&3&11&15&9&2\\\end{matrix}\\[2pt]
 	\textrm{\ \ \ \emph{Betti table of I}}&&&&&&&&\textrm{\ \ \ \emph{Betti table of P}}
 	\end{array}
 	\end{align*}
 	\begin{align*}
 	\arraycolsep=4.8pt\def\arraystretch{0.8}
 	\begin{array}{ccccccccc}
 	\begin{matrix}
 	&0&1\\
 	\text{total:}&3&2\\
 	\text{4:}&3&2 \\
 	\end{matrix}&&&&&&&&
 	\begin{matrix}
 	&0&1\\
 	\text{total:}&3&2\\
 	\text{5:}&3&2 \\
 	\end{matrix}
 	\\[4pt]
 	\textrm{\ \ \ \emph{Betti table of Q}} &&&&&&&& \textrm{\ \ \ \emph{Betti table of $P\cap Q$}}
 	\end{array}
 	\end{align*}
 	One can verify that $\beta_{i,j}(I)=\beta_{i,j}(P)+\beta_{i,j}(Q)+\beta_{i-1,j}(P\cap Q)$, for all $i, j$, \emph{i.e.}, $I=P+Q$ is a Betti splitting. Moreover, $\beta_{i,j}(I)= \beta_{i,j}(P)+\beta_{i,j}(Q)+\beta_{i-1,j-1}(Q)$, for all $i, j$.
 \end{Example}
 
 Next result will be crucial for the sequel.
 
 \begin{Proposition}\label{prop:CharactI=P+QLinear}
 	Let $I$ be a monomial ideal of $S$ generated in degree $d$. Suppose $I$ admits a Betti splitting $I=P+Q$. Then, the following conditions are equivalent:
 	\begin{enumerate}
 		\item[\rm(a)] $I$ has a $d$--linear resolution;
 		\item[\rm(b)] $P,Q$ have $d$--linear resolutions and $P\cap Q$ has a $(d+1)$--linear resolution.
 	\end{enumerate}
 \end{Proposition}
 \begin{proof}
 	(a) $\Rightarrow$ (b). It follows from \cite[Proposition 3.1]{DB}.\\
 	(b) $\Rightarrow$ (a). Let $\FF_P$, $\FF_Q$, $\FF_{P\cap Q}$ be the minimal graded free resolutions of the monomial ideals $P$, $Q$, $P\cap Q$, respectively. Furthermore, let $F_i^H$ be the $i$--th free module in the resolution $\FF_H$, $H\in\{P,Q,P\cap Q\}$. Since $I=P+Q$ is a Betti splitting, from (\ref{eq:BettiSplittingI=P+Q}), it follows that the $i$--th free module in the minimal graded free resolution of $I$ is the following one:
 	\begin{equation}\label{eq:PresenPQPcapQI}
 	F_i=F_i^P\oplus F_i^Q\oplus F_{i-1}^{P\cap Q},\,\, i\ge 0.
 	\end{equation}
 	On the other hand, by the hypothesis $P$ and $Q$ have $d$--linear resolutions, whereas $P\cap Q$ has a $(d+1)$--linear resolution. Thus $F_i^P= S(-i-d)^{\beta_i(P)}$, $F_i^Q= S(-i-d)^{\beta_i(Q)}$ and $F_{i-1}^{P\cap Q}= S(-(i-1)-(d+1))^{\beta_i(P\cap Q)} = S(-i-d)^{\beta_i(P\cap Q)}$, for all $i$. Hence, $F_i =S(-i-d)^{\beta_i(I)}$, for all $i$, and $I$ has a $d$--linear resolution.	
 \end{proof}
 
 In \cite{FHT2009}, some criteria to detect a Betti splitting were given. For later use, we recall the following one.
 \begin{Crit}\label{Lem:BettiSplitxiV2}
 	\textup{\cite[Corollary 2.7]{FHT2009}} Let $I$ be a monomial ideal of $S$. Let $P$ be the ideal generated
 	by all elements of $G(I)$ divisible by $x_i$ and let $Q$ be the ideal generated by all other elements of $G(I)$. If the ideal $P$ has a linear resolution, then $I=P+Q$ is a Betti splitting.
 \end{Crit}
 
 Next example \cite{HH2014} will describe a squarefree monomial ideal ($1$--spread ideal) with a Betti splitting and a linear resolution. For a non negative integer $n$, we set $[n]=\{1,\dots,n\}$. Given a non empty subset $A\subseteq[n]$, we set ${\bf x}_A=\prod_{i\in A}x_i$. Moreover, we set ${\bf x}_{\emptyset}=1$.
 
 \begin{Example} \rm Let $\Delta$ be a simplicial complex on the vertex set $[n]=\{1,\dots,n\}$. In \cite{HH2014}, Herzog and Hibi attached to $\Delta$ a special squarefree monomial ideal: let $R=K[{\bf x},{\bf y}]=K[x_1,\dots,x_n,y_1,\dots,y_n]$ be a polynomial ring in $2n$ variables with $K$ a field. Given a subset $F$ of $[n]$, define the squarefree monomial $u_F={\bf x}_F{\bf y}_{[n]\setminus F}=(\prod_{i\in F}x_i)(\prod_{i\in[n]\setminus F}y_i)$. Then, the \textit{face ideal} of $\Delta$ is the squarefree monomial ideal $J_\Delta=(u_F:F\in\Delta)$. Note that $J_{\Delta}$ is generated in degree $n$. In \cite[Corollary 1.2]{HH2014}, the authors proved that $J_\Delta$ has a $n$--linear resolution. By using the notion of Betti splitting one can recover such a property. We proceed by induction on $n$. For $n=1$, $\Delta=\{\emptyset,\{1\}\}$ and $J_{\Delta}=(u_{\emptyset},u_{\{1\}})=(x_1,y_1)$ has a linear resolution. Let $n>1$ and consider the simplicial complex on the vertex set $[n-1]$ defined by
 	$$
 	\Gamma=\big\{F\setminus\{n\}:F\in\Delta\big\}=\big\{F\in\Delta:n\notin F\big\}.
 	$$
 	We can write $J_\Delta=P+Q$, with $P=(w\in G(J_\Delta):x_n\ \textup{divides}\ w)$ and $Q$ such that $G(Q) = G(J_\Delta)\setminus G(P)$. Note that both $P=x_nJ_{\Gamma}$ and $Q=y_nJ_{\Gamma}$. By induction, $J_{\Gamma}$ has a $(n-1)$--linear resolution, thus $P$ and $Q$ have $n$--linear resolutions and $P\cap Q=x_ny_n J_{\Gamma}$ has an $(n+1)$--linear resolution. By Criterion \ref{Lem:BettiSplitxiV2}, $J_\Delta=P+Q$ is a Betti splitting, and by Proposition \ref{prop:CharactI=P+QLinear} we conclude that $J_{\Delta}$ has a $n$--linear resolution.
 \end{Example}
 
 Not all ideals with linear resolution admit a Betti splitting as \cite[Example 4.6]{DB} illustrates.
 
 \section{$t$--spread lexsegment ideals with linear resolution}\label{sec2}
 The purpose of this section is to classify all $t$--spread lexsegment ideals with a linear resolution.  We focus on the class of incompletely $t$--spread lexsegment ideals  since a characterization of all completely $t$--spread lexsegment ideals with a linear resolution was given in \cite[Theorem 4.4]{FC1}. The notion of Betti splitting will play a fundamental role.\medskip
 
 In order to keep the paper self--contained, we recall some comments and remarks from \cite{FC1}.
 Let $\mathcal{L}_t(u,v)$ be a $t$--spread lexsegment set and let $I=(\mathcal{L}_t(u,v))$. We may always assume
 that $u\ne v$. Indeed, a principal ideal has a linear resolution. We suppose also that $\deg(u)=\deg(v)>1$, otherwise $I$ is generated by variables and it has a linear resolution. Furthermore, we may assume that $x_1$ divides $u$. Indeed, if $\min(u)>1$, then the sequence $x_1,\dots,x_{\min(u)-1}$ is regular on $S/I$, and the graded Betti numbers of $I$ are the same of $I'=I\cap K[x_{\min(u)},x_{\min(u)+1},\dots,x_n]$. Hence, $I$ has a linear resolution if and only if $I'$ has a linear resolution. Moreover, we may assume $\min(u)<\min(v)$. Indeed, if $\min(u)=\min(v)$, then $I$ has a linear resolution if and only if $I''=(\mathcal{L}_t(u/x_{\min(u)},v/x_{\min(v)}))$ has a linear resolution. Finally, we assume that $I$ is not initial $t$--spread lexsegment neither final $t$--spread lexsegment because  in such cases $I$ has a linear resolution \cite[Theorem 1.4]{EHQ}.
 
 Criterion \ref{Lem:BettiSplitxiV2} allows us to state that there does exist a Betti splitting for a $t$--spread lexsegment ideal.
 
 \begin{Corollary}\label{Cor:IlinResNotCompletely}
 	Let $u,v\in M_{n,d,t}$ with $\min(u)=1$ and $\min(v)>1$ and $I=(\mathcal{L}_t(u,v))$ be a $t$--spread lexsegment ideal of $S$. Then there exists a Betti splitting of $I$. 
 \end{Corollary}
 \begin{proof} Set 
 	$$P=(\mathcal{L}_t(u,x_1x_{n-(d-2)t}\cdots x_{n-t}x_{n}))=x_1(\mathcal{L}_t^f(u/x_{1}))$$
 	and 
 	$$Q=(\mathcal{L}_t(x_2x_{2+t}\cdots x_{2+(d-1)t},v)).$$ 
 	One can observe that 
 	$P=(w\in G(I):x_1\ \textup{divides}\ w)$. It follows that $P$  has a linear resolution if and only if the ideal $(\mathcal{L}_t^f(u/x_1))$ considered in $K[x_2, \ldots,$ $x_n]$ has a linear resolution. Since $(\mathcal{L}_t^f(u/x_1))$ a $t$--spread final lexsegment ideal of the polynomial ring $K[x_2, \ldots,$ $x_n]$, the assertion follows from \cite[Theorem 1.4]{EHQ}. Hence, from Criterion \ref{Lem:BettiSplitxiV2}, $I=P+Q$ is a Betti splitting.
 \end{proof}
 
 \begin{Remark}\label{rem:PQlinear} \rm Note that the ideal $Q$ in Corollary \ref{Cor:IlinResNotCompletely} has a $d$--linear resolution. Indeed, $Q$ is an initial $t$--spread lexsegment of the polynomial ring $K[x_2, \ldots,x_n]$. Hence, Corollary \ref{Cor:IlinResNotCompletely} can be obtained also by \cite[Corollary 2.4]{FHT2009}.
 \end{Remark}
 
 \begin{Corollary}\label{Prop:BettiBol}
 	Let $u,v\in M_{n,d,t}$ with $\min(u)=1$ and $\min(v)>1$ and $I=(\mathcal{L}_t(u,v))$ be a $t$--spread lexsegment ideal of $S$. 
 	Set $P=(\mathcal{L}_t(u,x_1x_{n-(d-2)t}\cdots x_{n-t}x_{n}))$ and $Q=(\mathcal{L}_t(x_2x_{2+t}\cdots x_{2+(d-1)t},v))$.
 	The following statements are equivalent:
 	\begin{enumerate}
 		\item[\em(i)] $I$ has a $d$--linear resolution;
 		\item[\em(ii)] $P\cap Q$ has a $(d+1)$--linear resolution.
 	\end{enumerate}
 \end{Corollary}
 \begin{proof} It follows from Proposition \ref{prop:CharactI=P+QLinear} and the proof of Corollary \ref {Cor:IlinResNotCompletely}. Indeed, $P$ and $Q$ have $d$--linear resolutions.
 \end{proof}
 
 \begin{Remark}\rm Corollary \ref{Prop:BettiBol}, can be obtained without Proposition \ref{prop:CharactI=P+QLinear}. In fact, (i) $\Rightarrow $(ii) follows from \cite[Proposition 3.1]{DB}.
 	Conversely, assume that $P\cap Q$ has a $(d+1)$--linear resolution. Thus, $\reg(P\cap Q)=d+1$ and $P\cap Q$ is generated in degree $d+1$. Since $I=P+Q$ is a Betti splitting of $I$ (Corollary \ref{Cor:IlinResNotCompletely}), from \cite[Corollary 2.2 (a)]{FHT2009}, one has
 	\begin{equation}\label{eq:regmaxPQtLex}
 	\reg(I)=\max\{\reg(P),\reg(Q),\reg(P\cap Q)-1\}.
 	\end{equation}
 	On the other hand, $P$ and $Q$ have $d$--linear resolutions. Hence, $\reg(P)=\reg(Q)=d$ and from (\ref{eq:regmaxPQtLex}), it follows that $\reg(I)=d$. Finally, since $I$ is generated in degree $d$, $I$ has a $d$--linear resolution and (ii) $\Rightarrow $(i) follows.
 \end{Remark}
 \begin{Corollary}\label{Cor:i2>2+tNotCom}
 	Let $u=x_{i_1}x_{i_2}\cdots x_{i_d},v\in M_{n,d,t}$ with $\min(u)=1$ and $\min(v)>1$. Assume $I=(\mathcal{L}_t(u,v))$ has a $d$--linear resolution. If $I$ is an incompletely $t$--spread lexsegment ideal then $v\le_{\slex} x_2u/x_1$ and $i_2>2+t$.
 \end{Corollary}
 \begin{proof}
 	The proof of Corollary \ref{Cor:IlinResNotCompletely} provides a Betti splitting of $I$, namely $I = P+Q$, with $P=(\mathcal{L}_t(u,x_1x_{n-(d-2)t}\cdots x_{n-t}x_{n}))=x_1(\mathcal{L}_t^f(u/x_{1}))$ and
 	$ Q=(\mathcal{L}_t(x_2x_{2+t}\cdots$ $x_{2+(d-1)t}, v))$. Moreover, Corollary \ref{Prop:BettiBol} states that $I$ has a $d$--linear resolution if and only if $P\cap Q=x_1((\mathcal{L}_t^f(u/x_1))\cap Q$ has a $(d+1)$--linear resolution. Let $w\in G(P)=x_1\mathcal{L}_t^f(u/x_{1})$ and $z\in G(Q)=\mathcal{L}_t(x_2x_{2+t}\cdots x_{2+(d-1)t},v)$ monomials such that $\lcm(w,z)\in G(P\cap Q)$. Since $P\cap Q$ is generated in degree $d+1$, then $\lcm(w,z)=x_1z$ because $z$ is not divisible by $x_1$ and $\deg(\lcm(w,z))=d+1$. Hence, $w=x_1z/x_j$ for some $j\in\supp(z)$. In particular, $j\ge2$. It follows that $x_1z/x_j=w\in\mathcal{L}_t(u,x_1x_{n-(d-2)t}\cdots x_{n-t}x_{n})$, \emph{i.e.}, $x_1z/x_j\le_{\slex} u$ and consequently $v\le_{\slex} z\le_{\lex} x_ju/x_1\le_{\lex} x_2u/x_1$. Finally, $v\le_{\lex} x_2u/x_1$. Note that we have resorted to the use of the usual lexicographic order $\le_{\lex}$ because the monomials $x_ju/x_1$  ($j\ge 2$) may not be squarefree. \\
 	Now, we show that $i_2>2+t$. This implies that $v\le_{\slex} x_2u/x_1$. For our purpose, it is enough to show that if $i_2=1+t$ or $i_2=2+t$, then $I$ is a completely $t$--spread lexsegment ideal against our assumption. \\ Assume $i_2\in\{1+t,2+t\}$ and let $w\in M_{n,d,t}$, $w<_{\slex}v$. By the previous calculations, $v\le_{\lex} x_2u/x_1$, thus $w<_{\slex} v\le_{\lex} x_2u/x_1$. We distinguish two cases.\\
 	\textsc{Case 1.} Let $\min(w)=2$. Then, $w/x_2\le_{\slex} u/x_1$ and $x_1w/x_2\le_{\slex} u$.\\
 	\textsc{Case 2.} Let $\min(w)>2$. It follows that $\min(w/x_{\min(w)})>2+t\ge i_2=\min(u/x_1)$. Then, $w/x_{\min(w)}<_{\slex}u/x_1$ and $x_1w/x_{\min(w)}<_{\slex}u$.\\
 	Finally, in both cases we get $x_1w/x_{\min(w)}\le_{\slex} u$ and $I$ is a completely $t$--spread lexsegment ideal (Theorem \ref{thm:compltspreadlex}(i)).
 \end{proof}
 
 \begin{Corollary}\label{Cor:IlinResi2>1+tNotCom}
 	Let $u=x_1x_{i_2}x_{i_3}\cdots x_{i_d},v=x_{\ell}x_{n-(d-2)t}\cdots x_{n-t}x_{n}\in M_{n,d,t}$ for some $2\le\ell\le n-(d-1)t$ and $i_2\ge\ell+t$. Then $I=(\mathcal{L}_t(u,v))$ has a $d$--linear resolution.
 \end{Corollary}
 \begin{proof} As in Corollary \ref{Cor:IlinResNotCompletely}, set $P=(\mathcal{L}_t(u,x_1x_{n-(d-2)t}\cdots x_{n-t}x_{n}))=x_1(\mathcal{L}_t^f(u/x_{1}))$, and $Q=(\mathcal{L}_t(x_2x_{2+t}\cdots x_{2+(d-1)t},v)).$
 	Using Corollary \ref{Prop:BettiBol}, it is enough to prove that $P\cap Q$ has a $(d+1)$--linear resolution.
 	Note that $P\cap Q = x_1((\mathcal{L}_t^f(u/x_1))\cap Q)$. Indeed, let $w\in P\cap Q$. Then, $w=x_1ay=bz$, where $y\in \mathcal{L}_t^f(u/x_1)$, $z\in G(Q)$ and $a,b$ are monomials. Since $x_1$ does not divide $z$, then $x_1$ divides $b$ and therefore $w\in x_1((\mathcal{L}_t^f(u/x_1))\cap Q)$. The converse inclusion is clear.\\
 	Hence, $P\cap Q$ has a (d+1)--linear resolution if and only if the ideal $Y=(\mathcal{L}_t^f(u/x_1))\cap Q$ considered in $K[x_2,\dots,x_n]$ has a $d$--linear resolution. First, we prove that
 	\begin{equation}\label{eq:DecompY}
 	Y=(\mathcal{L}_t^f(u/x_1))\cap Q=\textstyle\sum\limits_{i=2}^{\ell}x_i(\mathcal{L}_t^f(u/x_1)).
 	\end{equation}
 	Let $w\in Y$. We have that $w=w_1y=w_2z$ with $y\in\mathcal{L}_t^f(u/x_1)$, $z\in G(Q)$ and $w_1,w_2$ monomials of $K[x_2,\dots,x_n]$. Let $i=\min(z)$. Then, $2\le i\le\ell$ and $\min(y)\ge i_2\ge\ell+t$, for some $2\le\ell\le n-(d-1)t$. It follows that  $x_i$ divides $w=w_1y$ and does not divide $y$. Thus $x_i$ divides $w_1$. We have shown that $w$ is a multiple of $x_iy\in x_i(\mathcal{L}_t^f(u/x_1))$, $2\le i \le \ell$. Hence, $Y\subseteq\sum_{i=2}^{\ell}x_i(\mathcal{L}_t^f(u/x_1))$.\\
 	Conversely, let $w\in\sum_{i=2}^{\ell}x_i(\mathcal{L}_t^f(u/x_1))$. Then, $w=x_iy$, for some $2\le i \le \ell$. Note that $w$ is $t$--spread. Indeed, $\min(y)\ge i_2\ge \ell+t$, $i\le\ell$ and so $\min(y)-i\ge\min(y)-\ell\ge t$. On the other hand, $x_2x_{2+t}\cdots x_{2+(d-1)t}\ge_{\slex} w$ and $w\ge_{\slex} v$. In fact, $\min(w)=i\le\ell=\min(v)$ and $v$ is the smallest $t$--spread monomial with minimum equal to $\ell$. Hence $w\in\mathcal{L}_t(x_2x_{2+t}\cdots x_{2+(d-1)t},v)=G(Q)$. Finally, $w\in(\mathcal{L}_t^f(u/x_1))\cap Q=Y$.\\
 	Now, we prove that $Y$ has a $d$--linear resolution. We proceed by induction on $\ell$.
 	Let $\ell=2$. By (\ref{eq:DecompY}), $Y=x_2(\mathcal{L}_t^f(u/x_1))$ has a $d$--linear resolution because $(\mathcal{L}_t^f(u/x_1))$ has a $(d-1)$--linear resolution.
 	Let $\ell>2$, and set
 	\begin{align*}
 	Y_1&=\textstyle\sum\limits_{i=2}^{\ell-1}x_i(\mathcal{L}_t^f(u/x_1)),\ \ \ \ \ \ \ \
 	Y_2=x_\ell(\mathcal{L}_t^f(u/x_1)).
 	\end{align*}
 	By induction, both $Y_1$ and $Y_2$ have $d$--linear resolutions. Note that $Y=Y_1+Y_2$ is a Betti splitting. Indeed, $Y_2=(w\in G(Y):x_{\ell}\ \textup{divides}\ w)$ has a $d$--linear resolution. By Proposition \ref{prop:CharactI=P+QLinear}, $Y$ has a $d$--linear resolution if and only if $Y_1\cap Y_2$ has a $(d+1)$--linear resolution. As $x_\ell$ does not divide any minimal generator of $Y_1$, we have that
 	\begin{align*}
 	Y_1\cap Y_2\ &=\ \textstyle\sum\limits_{i=2}^{\ell-1}x_i(\mathcal{L}_t^f(u/x_1))\cap x_{\ell}(\mathcal{L}_t^f(u/x_1))=\textstyle\sum\limits_{i=2}^{\ell-1}x_ix_{\ell}(\mathcal{L}_t^f(u/x_1))\\
 	&=\ x_{\ell}\big(\textstyle\sum\limits_{i=2}^{\ell-1}x_i(\mathcal{L}_t^f(u/x_1))\big)=x_{\ell}Y_1.
 	\end{align*}
 	Since $Y_1$ has a $d$--linear resolution, then $Y_1\cap Y_2=x_\ell Y_1$ has a $(d+1)$--linear resolution. The assertion follows.
 \end{proof}
 
 In order to get our main result we need to prove another statement.
 \begin{Proposition}\label{prop:suffconditionIcompTSpread}
 	Let $u,v\in M_{n,d,t}$ with $\min(u)=1$ and $\min(v)>1$. Suppose that $\min(u/x_1)<\min(v)+t$. Then $I=(\mathcal{L}_t(u,v))$ is a completely $t$--spread lexsegment ideal.
 \end{Proposition}
 \begin{proof}
 	For any $w=x_{k_1}x_{k_2}\cdots x_{k_d}\in M_{n,d,t}$ with $w<_{\slex}v$, we have that $\min(w)=k_1\ge\min(v)$ and $k_2\ge\min(v)+t$. By the hypothesis, $\min(u/x_1)<\min(v)+t$.  Hence, from the inequality $\min(v)+t\le k_2 = \min(w/x_{\min(w)})$, it follows that $u/x_1>_{\slex}w/x_{\min(w)}$,\emph{i.e.}, $u >_{\slex}x_1w/x_{\min(w)}$. Finally, from Theorem \ref{thm:compltspreadlex}(i), $I$ is a completely $t$--spread lexsegment ideal.
 \end{proof}
 
 We are in position to prove our main theorem. It deals with a pivotal case in the classification of the $t$--spread lexsegment ideals with linear resolution.
 \begin{Theorem}\label{teor:InotCompLinRes}
 	Let $u=x_{i_1}x_{i_2}\cdots x_{i_d},v=x_{j_1}x_{j_2}\cdots x_{j_d}\in M_{n,d,t}$ with $\min(u)=i_1=1$ and $\min(v)=j_1>1$. Assume $I=(\mathcal{L}_t(u,v))$ is an incompletely $t$--spread lexsegment ideal. Then $I$ has a linear resolution if and only if
 	$
 	u<_{\slex}x_1x_{\ell+1+t}\cdots x_{\ell+1+(d-1)t}$ and $v=x_{\ell}x_{n-(d-2)t}\cdots x_{n-t}x_n$ for some $2\le\ell<n-(d-1)t$.
 \end{Theorem}
 \begin{proof}
 	Assume $I$ has a linear resolution. Set $j_1=\ell$. One has $2\le\ell\le n-(d-1)t$. The case $\ell=n-(d-1)t$ is not admissible, otherwise $I$ should be a final $t$--spread lexsegment ideal and thus a completely $t$--spread lexsegment ideal. Therefore, we distinguish two cases: $\ell=2$ and $2<\ell<n-(d-1)t$.\\
 	From Corollary \ref{Cor:IlinResNotCompletely}, the $t$--spread lexsegment ideal $I = P+Q$ is a Betti splitting, with $P=$ $(\mathcal{L}_t(u,x_1x_{n-(d-2)t}\cdots x_{n-t}x_{n}))=x_1(\mathcal{L}_t^f(u/x_{1}))$,
 	$Q=(\mathcal{L}_t(x_2x_{2+t}\cdots x_{2+(d-1)t},v))$. \\
 	\textsc{Case 1.} Let $\ell=2$. By Corollary \ref{Cor:i2>2+tNotCom}, $i_2>2+t=\ell+t$.
 	First, we prove that $v=x_2x_{n-(d-2)t}\cdots x_{n-t}x_{n}$. \\
 	Since $I$ has a $d$--linear resolution, then  $P\cap Q$ has a $(d+1)$--linear resolution (Corollary \ref{Prop:BettiBol}). Suppose for absurd that $v\ne x_{2}x_{n-(d-2)t}\cdots x_{n-t}x_{n}$. By Corollary \ref{Cor:i2>2+tNotCom}, we have $v\le_{\slex} x_2u/x_1$, thus $v/x_2\le_{\slex} u/x_1$ and $j_2\ge i_2>2+t$. Consider the monomial $w=x_2x_{2+t}x_{n-(d-3)t}\cdots x_{n-t}x_{n}$. $w$ is $t$--spread and $w>_{\slex}v$. In fact, $\min(w)=\min(v)=2$ and $2+t<j_2$. Thus, $w\in G(Q)=\mathcal{L}_t(x_2x_{2+t}\cdots x_{2+(d-1)t},v)$. Consider the monomial $z=x_1x_{n-(d-2)t}\cdots x_{n-t}x_{n}\in G(P)=\mathcal{L}_t(u,x_1x_{n-(d-2)t}\cdots x_{n-t}x_{n})$. It follows that
 	$$
 	y=\lcm(w,z)=x_1x_2x_{2+t}\big(\textstyle\prod_{p=0}^{d-2}x_{n-pt}\big)\in P\cap Q, 
 	$$
 	with $\deg(y)=d+2$. We will show that $y\in G(P\cap Q)$. As a consequence, we get that $P\cap Q$ is not generated in degree $d+1$ and so  $I$ will not  have a linear resolution.
 	All monomials of $P\cap Q$ are divided by $x_1x_2$. Hence, if $y$ is not a minimal generator of $P\cap Q$, either $y/x_{2+t}\in P\cap Q$ or $y/x_{n-pt}\in P\cap Q$, for some $p\in\{0,\dots,d-2\}$. The first instance does not occur, otherwise $y/x_{2+t}$ should be equal to $\lcm(z,x_2x_{n-(d-2)t}\cdots x_{n-t}x_{n})$. But our assumption on $v$ implies that $x_2x_{n-(d-2)t}\cdots x_{n-t}x_{n}\notin Q$. The second instance, namely $y/x_{n-pt}\in P\cap Q$, for some $p\in\{0,\dots,d-2\}$, does not occur either. Indeed, in such a case $y/(x_2x_{n-pt})=x_1x_{k_2}\cdots x_{k_d}\in G(P)$, necessarily. But $k_2=2+t<i_2=\min(u/x_1)$ and so $y/(x_2x_{n-pt})>_{\slex}u$ and $y/(x_2x_{n-pt})\notin P$. Hence, $I$ should not have a linear resolution (Theorem \ref{thm:compltspreadlex}(ii)). Therefore, it must be $v=x_{2}x_{n-(d-2)t}\cdots x_{n-t}x_{n}$.\\
 	Now, we determine all admissible monomials $u\in M_{n,d,t}$ such that $I$ is an incompletely $t$--spread lexsegment ideal. We will verify that it must be $$u<_{\slex}x_{1}x_{3+t}\cdots x_{3+(d-1)t}.$$ Consider the $t$--spread lexsegment set
 	$$\Omega=\{w\in M_{n,d,t}:w<_{\slex}v\}= \mathcal{L}_t(x_3x_{3+t}\cdots x_{3+(d-1)t},x_{n-(d-1)t}\cdots x_{n-t}x_{n}).$$
 	Suppose $u\ge_{\slex} x_1x_{3+t}\cdots x_{3+(d-1)t}$. Then, for any $w=x_{k_1}x_{k_2}\cdots x_{k}\in\Omega$, we have $k_1\ge3$, $k_2\ge 3+t$, $\dots$, $k_d\ge3+(d-1)t$ and consequently $u\ge_{\slex} x_1(w/x_{k_1})$. Hence, from Theorem \ref{thm:compltspreadlex}(i), $I$ should be a completely $t$--spread lexsegment ideal, contradicting the assumption on $I$.
 	Hence,  $u<_{\slex}x_1x_{3+t}\cdots x_{3+(d-1)t}$.\\ Indeed, if we assume $u<_{\slex}x_1x_{3+t}\cdots x_{3+(d-1)t}$, then $x_1(x_3x_{3+t}\cdots x_{3+(d-1)t})/x_{3+st}$ $>_{\slex}u$ for all $s=0,\dots,d-1$. By Theorem \ref{thm:compltspreadlex}(i), it follows that $I$ is an incompletely $t$--spread lexsegment ideal. The claim follows.\\
 	Conversely, if $v=x_{2}x_{n-(d-2)t}\cdots x_{n-t}x_{n}$ and $u<_{\slex}x_1x_{3+t}\cdots x_{3+(d-1)t}$, then $I$ has a linear resolution by Corollary \ref{Cor:IlinResi2>1+tNotCom}.\medskip\\
 	\textsc{Case 2.} Let $2<\ell<n-(d-1)t$. Since $I$ is an incompletely $t$--spread lexsegment ideal, Proposition \ref{prop:suffconditionIcompTSpread} implies $i_2\ge\ell+t$. Let $\widetilde{S}$ be the polynomial ring $K[x_1,x_{\ell},x_{\ell+1},\dots,x_n]$ and denote by $\widetilde{I}$ the ideal $I\cap\widetilde{S}$. It follows that $\widetilde{I}$ is a monomial ideal of $\widetilde{S}$ such that $G(\widetilde{I})=\{w\in G(I):\supp(w)\subseteq\{1,\ell,\ell+1,\dots,n\}\}$. Since $i_2\ge\ell+t$, $\widetilde{I}$ is again a $t$--spread lexsegment ideal of $\widetilde{S}$ which is an incompletely $t$--spread lexsegment. In fact, the sets $\{w\in M_{n,d,t}:w<_{\slex}v\}$ and $\{w\in M_{n,d,t}\cap\widetilde{S}:w<_{\slex}v\}$ are the same, as any monomial $w\in M_{n,d,t}$ with $w<_{\slex}v$ is such that $\min(w)\ge\min(v)=\ell$. Hence, $w\in\widetilde{S}$. On the other hand, $u\in\widetilde{S}$  and so  $I$ is a completely $t$--spread lexsegment ideal if and only if $\widetilde{I}$ is completely $t$--spread lexsegment. Thus, $\widetilde{I}$ is an incompletely $t$--spread lexsegment ideal since $I$ is incompletely $t$--spread lexsegment. 
 	\\
 	We claim that if $I$ has a $d$--linear resolution, then $\widetilde{I}$ has a $d$--linear resolution, too. For this purpose, note that $\mathcal{L}_t(u,x_1x_{n-(d-2)t}\cdots x_{n-t}x_{n})\subseteq G(\widetilde{I})$ and
 	$$
 	G(I)\setminus G(\widetilde{I})=\mathcal{L}_t(x_2x_{2+t}\cdots x_{2+(d-1)t},x_{\ell-1}x_{n-(d-2)t}\cdots x_{n-t}x_{n}).
 	$$
 	Let $J$ be the ideal generated by the $t$--spread lexsegment $G(I)\setminus G(\widetilde{I})$. If we think of $\widetilde{I}$ as an ideal of $S$, it is clear that $I=\widetilde{I}+J$. Hence, we can consider the natural short exact sequence of $K$--vector spaces
 	$$
 	0\rightarrow \widetilde{I}\cap J\rightarrow\widetilde{I}\oplus J\rightarrow I\rightarrow0
 	$$
 	which induces for every ${\bf a}=(a_1,\dots,a_n)\in\ZZ^n$ the long exact sequence
 	\begin{equation}\label{eq:LongExactSequenceTorbfa}
 	\cdots\rightarrow\Tor_i^S(K,\widetilde{I}\cap J)_{\bf a}\rightarrow\Tor_i^S(K,\widetilde{I})_{\bf a}\oplus\Tor_i^S(K,J)_{\bf a}\rightarrow\Tor_i^S(K,I)_{\bf a}\rightarrow\cdots,
 	\end{equation}
 	where the lower index ${\bf a}$ denotes the multigraded ${\bf a}$--th component.
 	
 	By contradiction, suppose that the minimal free resolution of $\widetilde{I}$ is not linear. Then there exists an $n$--tuple ${\bf a}=(a_1,\dots,a_n)\in\ZZ^n$ and an integer $i$ such that $\sum_{j=1}^na_j>d-i$ and $\Tor_i^S(K,\widetilde{I})_{\bf a}\ne0$. Since $\widetilde{I}$ is a monomial ideal of $\widetilde{S}=K[x_1,x_{\ell},x_{\ell+1},\dots,x_{n}]$, then 
 	$\Tor_i^S(K,\widetilde{I})_{\bf c}= 0$ if $c_j\neq 0$ for some $j\in \{2, \ldots, \ell\}$. Therefore, $a_j=0$ for $j=2,\dots,\ell-1$. Furthermore, all generators of $J$ and $\widetilde{I}\cap J$ are divided by at least one of the variables $x_2,x_3,\dots,x_{\ell-1}$, and thus $\Tor_i^S(K,J)_{\bf c}=\Tor_i^S(\widetilde{I}\cap J)_{\bf c}=0$ for all $i$ and all ${\bf c}=(c_1,\dots,c_n)\in\ZZ^n$ with $c_j=0$ for $j=2,\dots,\ell-1$. Hence, from the long exact sequence (\ref{eq:LongExactSequenceTorbfa}), we have $\Tor_i^S(K,I)_{\bf a}\cong\Tor_i^S(K,\widetilde{I})_{\bf a}\ne0$, against the fact that $I$ has a $d$--linear resolution. It follows that if $I$ has a linear resolution, then $\widetilde{I}$ has a $d$--linear resolution too. By \textsc{Case 1} applied to $\widetilde{I}$, we get $v=x_{\ell}x_{n-(d-2)t}\cdots x_{n-t}x_{n}$ and  $u<_{\slex}x_1x_{\ell+1+t}\cdots x_{\ell+1+(d-1)t}$.\\
 	Conversely, if $v=x_{\ell}x_{n-(d-2)t}\cdots x_{n-t}x_{n}$, $u<_{\slex}x_1x_{\ell+1+t}\cdots x_{\ell+1+(d-1)t}$, then $I$ has a $d$--linear resolution by Corollary \ref{Cor:IlinResi2>1+tNotCom}.
 \end{proof}
 
 \section{Some applications}\label{sec3}
 
 This section contains some applications of the results obtained in the previous sections. More precisely, we compute the Betti numbers of incompletely $t$--spread lexsegment ideals with linear resolutions and, furthermore,  we characterize all incompletely $t$--spread lexsegment ideal with linear quotients.
 
 In order to describe the Betti numbers of the class of $t$--spread lexsegment ideals, we have included \cite[Proposition 4.5]{FC1} as statement (a) in the following theorem.
 
 \begin{Theorem}\label{thm:Bettiformula}
 	Let $u,v\in M_{n,d,t}$ with $\min(u)=1$, $\min(v)>1$ and let $I=(\mathcal{L}_t(u,v))$. Assume $I$ has a linear resolution.
 	\begin{enumerate}
 		\item[\textup{(a)}] If $I$ is a completely $t$--spread lexsegment ideal, then for all $i\ge0$
 		$$
 		\beta_{i}(I)=\sum_{w\in\mathcal{L}_t^f(u)}\binom{n-\min(w)-(d-1)t}{i}-\sum_{\substack{w\in\mathcal{L}_t^f(v)\\ w\ne v}}\binom{\max(w)-(d-1)t-1}{i}.
 		$$
 		\item[\textup{(b)}] If $I$ is an incompletely $t$--spread lexsegment ideal, then for all $i\ge0$
 		\begin{align*}
 		\beta_i(I)&=\sum_{\substack{w\in\mathcal{L}_t(u,v)\\ \min(w)\ge2}}\binom{\max(w)-(d-1)t-2}{i}\\
 		&+\sum_{j=0}^{i}\sum_{w\in\mathcal{L}_t^f(u/x_1)}\binom{\min(v)-1}{j}\binom{n-\min(w)-(d-2)t}{i-j}.
 		\end{align*}
 	\end{enumerate}
 \end{Theorem}
 \begin{proof}
 	(a) See Theorem \ref{thm:compltspreadlex}(iii). \\
 	(b) As shown in Corollary \ref{Cor:IlinResNotCompletely}, $I=P+Q$ is a Betti splitting with $P=(\mathcal{L}_t(u,x_1x_{n-(d-2)t}$ $\cdots x_{n-t}x_{n}))$ $=x_1(\mathcal{L}_t^f(u/x_{1}))$ and $Q=(\mathcal{L}_t(x_2x_{2+t}\cdots x_{2+(d-1)t},v))$. Moreover, it follows from Corollary \ref{Prop:BettiBol} and the proof of Corollary \ref{Cor:IlinResi2>1+tNotCom} that
 	$P\cap Q=x_1Y$ has a $(d+1)$--linear resolution, where $Y=(\mathcal{L}_t^f(u/x_1))\cap Q=\textstyle\sum\limits_{i=2}^{\min(v)}x_i(\mathcal{L}_t^f(u/x_1))$. 
 	
 	Furthermore, from Corollary \ref{Cor:IlinResi2>1+tNotCom} and (\ref{eq:DecompY}), we have 
 	\begin{equation}\label{eq:Betti}
 	P\cap Q=x_1(x_2,\dots,x_{\min(v)})(\mathcal{L}_t^f(u/x_1))=(x_2,\dots,x_{\min(v)})P.
 	\end{equation} 
 	Indeed, $\supp(P)\cap\{2,\dots,{\min(v)}\}=\emptyset$, $(x_2,\dots,x_{\min(v)})\cap P=(x_2,\dots,x_{\min(v)})P$, where $\supp(P)=\bigcup_{w\in G(P)}\supp(w)$. Since $I=P+Q$ is a Betti splitting, then
 	$$\beta_{i}(I)=\beta_{i}(P)+\beta_{i}(Q)+\beta_{i-1}(P\cap Q),$$ 
 	for every $i\ge 0$. On the other hand, $P=x_1(\mathcal{L}_t^f(u/x_1))$ and so $\beta_{i}(P)=\beta_{i}((\mathcal{L}_t^f(u/x_1)))$. Moreover, $(\mathcal{L}_t^f(u/x_1))$ is a $t$--spread strongly stable ideal of the polynomial ring $S$ with the order on the variables reversed ($x_n>x_{n-1}>\dots>x_1$), whereas $Q$ is a $t$--spread strongly stable ideal of the polynomial ring $K[x_2,\dots,x_n]$ in $n-1$ variables with the usual order on the variables $x_2>x_3>\dots>x_n$. Hence, in order to compute $\beta_i(Q)$, we must substitute in (\ref{eq1}) $\max(w)$ with $\max(w)-1$, for every $w\in G(Q)$; whereas when we compute $\beta_i((\mathcal{L}_t^f(u/x_1)))$ by formula (\ref{eq2}), we must take into account that $(\mathcal{L}_t^f(u/x_1))$ is equigenerated in degree $d-1$. 
 	
 	Finally, setting $\ell=\min(v)$ and applying \cite[Proposition 3.1]{HRR} (see, also, \cite[Corollary 3.5]{HG}), from (\ref{eq:Betti}), we have:
 	\begin{align*}
 	\beta_{i-1}(P\cap Q)\ &=\ \sum_{k=0}^{i-1}\beta_{j}((x_2,\dots,x_\ell))\cdot\beta_{i-1-k}(P)\\
 	&=\ \sum_{k=0}^{i-1}\binom{\ell-1}{k+1}\beta_{i-k-1}(P).
 	\end{align*}
 	If we set $j=k+1$, then $j=1,\dots,i$ and we can write
 	\begin{align*}
 	\beta_i(P)+\beta_{i-1}(P\cap Q)\ &=\ \beta_i(P)+\sum_{j=1}^{i}\binom{\ell-1}{j}\beta_{i-j}(P)\\
 	&=\ \sum_{j=0}^i\binom{\ell-1}{j}\beta_{i-j}(P)\\
 	&=\ \sum_{j=0}^{i}\sum_{w\in\mathcal{L}_t^f(u/x_1)}\binom{\min(v)-1}{j}\binom{n-\min(w)-(d-2)t}{i-j}.
 	\end{align*}
 	Condition (b) follows.
 \end{proof}

 We conclude the section with our last application which is related to the determination of incompletely $t$--spread lexsegment ideals which have linear quotients.
 
 Recall that a monomial ideal $I\subset S$ has linear quotients if there exists an ordering $u_1>u_2>\dots>u_m$ of its minimal generating set $G(I)$ such that  the colon ideals $(u_1,\dots,u_{k-1}):u_k$ are generated by a subset of the set of the variables $\{x_1, \ldots, x_n\}$, for all $k=2,\dots,m$. If $I$ is a monomial ideal equigenerated in degree $d$ and $I$ has linear quotients, then $I$ has a $d$--linear resolution \cite[Proposition 8.2.1]{JT}. The converse is not true. However, if $I$ is a monomial ideal generated in degree $2$, then $I$ has linear quotients if and only if $I$ has a linear resolution \cite[Theorem 10.2.6]{JT}. Thus, one can ask for what other classes of monomial ideals these properties are equivalent.
 
 In the following, we show that the aforementioned properties are equivalent for all incompletely $t$--spread lexsegment ideals. 
 \begin{Theorem}\label{thm:App2}
 	Let $I$ be an incompletely $t$--spread lexsegment ideal generated in degree $d$. Then $I$ has linear quotients if and only if $I$ has a linear resolution.
 \end{Theorem}
 \begin{proof}
 	Since $I$ is generated in degree $d$, then if $I$ has linear quotients, it follows that $I$ has a linear resolution. Conversely, suppose $I=(\mathcal{L}_t(u,v))$ has a linear resolution, $\min(u)=1$ and $\min(v)>1$. By Theorem \ref{teor:InotCompLinRes}, we have
 	$u<x_1x_{\ell+1+t}\cdots x_{\ell+1+(d-1)t}$ and $v=x_{\ell}x_{n-(d-2)t}\cdots x_{n-t}x_n,
 	$
 	for some $2\le\ell<n-(d-1)t$. Moreover, from the proof of Corollary \ref{Cor:IlinResi2>1+tNotCom}, if we set $P=(\mathcal{L}_t(u,x_1x_{n-(d-2)t}\cdots x_{n-t}x_{n}))=x_1(\mathcal{L}_t^f(u/x_{1}))$ and $Q=(\mathcal{L}_t(x_2x_{2+t}\cdots x_{2+(d-1)t},v))$, then $P\cap Q=(x_2,\dots,x_\ell)P$.\\
 	Following \cite{EOS}, we order the generators of $I$ in a suitable way. Firstly, note that $G(I)$ is the disjoint union of $G(P)$ and $G(Q)$. Let $>_{\lex}$ be the lexicographic order corresponding to $x_1>x_2>\dots>x_n$ and let $>_{\overline{\lex}}$ be the lexicographic order corresponding to $x_n>x_{n-1}>\dots>x_1$. We order $G(Q)=\{w_1>_{\lex}w_2>_{\lex}\dots>_{\lex}w_p\}$ decreasingly with respect to $>_{\lex}$ and $G(P)=\{z_1>_{\overline{\lex}}z_2>_{\overline{\lex}}\cdots>_{\overline{\lex}}z_q\}$ decreasingly with respect to $>_{\overline{\lex}}$. Then, we order $G(I)$ as follows:
 	\[
 	w_1>_{\lex} w_2>_{\lex} \cdots >_{\lex} w_p >_{\overline{\lex}} z_1>_{\overline{\lex}} z_2>_{\overline{\lex}} \cdots  >_{\overline{\lex}} z_q.
 	\]
 	In order to simplify the notation, we will denote such an order by $\succ$. Thus
 	\begin{equation}\label{eq:orderingG(I)G(P)G(Q)}
 	w_1\succ w_2\succ\dots\succ w_p\succ z_1\succ z_2\succ\cdots\succ z_q.
 	\end{equation}
 	The colon ideal $(w_1,\dots,w_{k-1}):w_k$ is generated by variables, for all $k=2,\dots,p$. Indeed, $Q$ is an initial $t$--spread lexsegment ideal in $K[x_2,\dots,x_n]$. Thus, it is also a $t$--spread strongly stable ideal and so it has linear quotients with respect to the ordering $w_1\succ w_2\succ\dots\succ w_p$ \cite[Theorem 1.4]{EHQ}. If one take $k\in\{1,\dots,q\}$, then
 	\begin{align*}
 	(w_1,\dots,w_{p},z_1,\dots,z_{k-1}):z_{k}\ &=\ (w_1,\dots,w_{p}):z_{k}+(z_1,\dots,z_{k-1}):z_{k}\\&=\ Q:z_{k}+(z_1,\dots,z_{k-1}):z_{k}.
 	\end{align*}
 	Since $P$ can be seen as a $t$--spread strongly stable ideal but with the order on the variables reversed ($x_n>x_{n-1}>\dots>x_1$), $P$ has linear quotients with respect to the ordering $z_1\succ z_2\succ\cdots\succ z_q$  \cite[Theorem 1.4]{EHQ}. Thus, $(z_1,\dots,z_{k-1}):z_{k}$ is generated by variables. Thus, if we show that $Q:z_k$ is generated by variables too, we can conclude that also $(w_1,\dots,w_{p},z_1,\dots,z_{k-1}):z_{k}$ is generated by variables, as desired. \\
 	Let us prove that $Q:z_k = (x_2, \ldots, x_\ell)$. Firstly, one can observe that
 	$$
 	z_k(Q:z_{k})=Q\cap(z_k)\subseteq Q\cap P=(x_2,\dots,x_\ell)P.
 	$$
 	Let $y\in Q:z_{k}$. Then, $yz_k\in(x_2,\dots,x_\ell)P$ and $x_jz_{s}$ divides $yz_{k}$, for some $j\in\{2,\dots,\ell\}$ and some $s\in\{1,\dots,q\}$. Since $z_{k}\le_{\slex} u<_{\slex}x_{1}x_{\ell+1+t}\cdots x_{\ell+1+(d-1)t}$, then  $x_j$ does not divide $z_{k}$. Hence, $x_j$ divides $y$. This shows that $Q:z_{k}\subseteq (x_2,\dots,x_{\ell})$. Conversely, we prove that $(x_2,\dots,x_{\ell})\subseteq Q:z_{k}$. Indeed, if one picks $x_j\in(x_2,\dots,x_\ell)$, then $x_jz_{k}\in Q$. In fact, since $z_{k}\le_{\lex} u<_{\slex}x_{1}x_{\ell+1+t}\cdots x_{\ell+1+(d-1)t}$ and $\min(z_{k})=1$, we have that $\min(z_k/x_1)\ge\ell+1+t$. Thus $\min(z_k/x_1)-j>t$ and $x_j(z_{k}/x_1)$ is $t$--spread. Furthermore,  $x_j(z_{k}/x_1)\in \mathcal{L}_t(x_2x_{2+t}\cdots x_{2+(d-1)t},v) = G(Q)$. Thus $x_1x_j(z_{k}/x_1)=x_jz_{k}\in Q$. The assertion follows.
 \end{proof}
 
 \begin{Remark}
 	\rm Let $I=(\mathcal{L}_t(u,v))=P+Q$ be an incompletely $t$--spread lexsegment ideal with $P$ and $Q$ as in Theorem \ref{thm:App2}. Let $G(Q)=\{w_1>_{\lex}w_2>_{\lex}\dots>_{\lex}w_p\}$ ordered decreasingly with respect to $>_{\lex}$ and $G(P)=\{z_1>_{\overline{\lex}}z_2>_{\overline{\lex}}\cdots>_{\overline{\lex}}z_q\}$ ordered decreasingly with respect to $>_{\overline{\lex}}$. Define
 	\begin{align*}
 	\set_Q(w_k)&=\{x_i:x_i\in (w_1,\dots,w_{k-1}):w_k\},\\
 	\set_P(z_k)&=\{x_i:x_i\in (z_1,\dots,z_{k-1}):z_k\}.
 	\end{align*}
 	Setting 
 	\begin{align*}
 	\set_I(w_k)&=\{x_i:x_i\in (w_1,\dots,w_{k-1}):w_k\},\\
 	\set_I(z_k)&=\{x_i:x_i\in (w_1,\dots,w_{p},z_1,\dots,z_{k-1}):z_k\},
 	\end{align*}
 	Theorem \ref{thm:App2} points out that if we order the monomials of $G(I)$ as in (\ref{eq:orderingG(I)G(P)G(Q)}), then $\set_I(w_k)=\set_Q(w_k)$ and $\set_I(z_k)=\{x_2,\dots,x_\ell\}\cup\set_P(z_k)$, where the union is disjoint.\\
 	By \cite[Corollary 8.2.2]{JT},
 	\begin{equation}\label{eq:set}
 	\beta_{i}(I)=\sum_{y\in G(I)}\binom{|\set_I(y)|}{i}.
 	\end{equation}
 	We claim that we can recover the formula given in Theorem \ref{thm:Bettiformula}(b).
 	Indeed, from formulas (\ref{eq1}) and (\ref{eq2}), we have that $|\set_Q(w_k)|=\max(w_k)-2-(d-1)t$ and $|\set_P(z_k)|=n-\min(z_k)-(d-2)t$. Hence,
 	\begin{align*}
 	\beta_i(I)&=\sum_{y\in G(I)}\binom{|\set_I(y)|}{i}=\sum_{w\in G(Q)}\binom{|\set_Q(w)|}{i}+\sum_{z\in G(P)}\binom{|\{x_2,\dots,x_{\ell}\}\cup\set_P(z)|}{i}\\[9pt]
 	&=\sum_{\substack{w\in\mathcal{L}_t(u,v)\\ \min(w)\ge2}}\!\!\binom{\max(w)-(d-1)t-2}{i}+\sum_{z\in G(P)}\sum_{j=0}^{i}\binom{|\{x_2,\dots,x_{\ell}\}|}{j}\binom{|\set_P(z)|}{i-j}\\[9pt]
 	&=\sum_{\substack{w\in\mathcal{L}_t(u,v)\\ \min(w)\ge2}}\!\!\binom{\max(w)-(d-1)t-2}{i}\!+\!\!\!\!\!\!\sum_{\substack{j=0,\ldots,i\\w\in\mathcal{L}_t^f(u/x_1)}}\!\!\!\!\binom{\ell-1}{j}\binom{n-\min(w)-(d-2)t}{i-j},
 	\end{align*}
 	which coincides with the formula given in Theorem \ref{thm:Bettiformula}(b).
 \end{Remark}
 
 We illustrate the previous results of this section with an example.
 \begin{Example}
 	\rm Let $S=K[x_1,\dots,x_8]$ and consider the $2$--spread lexsegment ideal $I=(\mathcal{L}_t(x_1x_5x_8,x_3x_6x_8))$ of $S$. Then
 	\[
 	I = (x_1x_5x_8,x_1x_6x_8,x_2x_4x_6,x_2x_4x_7,x_2x_4x_8,x_2x_5x_7,x_2x_5x_8,x_2x_6x_8)
 	\]
 	and it is an incompletely $2$--spread lexsegment ideal with a linear resolution (Theorem \ref{thm:compltspreadlex}). Indeed, setting $u=x_1x_5x_8$ and $v=x_2x_6x_8$, $n=8$, $d=3$, $t=2$, the monomials $w\in M_{8,3,2}$ with $w<_{\slex}v$ are $x_3x_5x_7,x_3x_5x_8,x_3x_6x_8,x_4x_6x_8$. Consider the monomial $w=x_3x_5x_7$. Note that $x_1w/x_s>_{\slex}u$, for all $s\in\supp(w)$ and thus $I$ is an incompletely $2$--spread lexsegment ideal. With the same notations as in Theorem \ref{teor:InotCompLinRes}, let $\ell=\min(v)=2$. Then $\ell\in\{2,\dots,n-(d-1)t-1\}=\{2,3\}$, $v=x_{\ell}x_{n-(d-2)t}\cdots x_{n}=x_2x_6x_8$ and $u<_{\slex}x_1x_{\ell+1+t}\cdots x_{\ell+1+(d-1)t}=x_1x_5x_7$. Hence, from Theorem \ref{teor:InotCompLinRes}, $I$ has a linear resolution. In fact, by \textit{Macaulay2} \cite{GDS}, 
 	\[0 \longrightarrow S(-5)^4 \longrightarrow S(-4)^{11} \longrightarrow S(-3)^{8} \longrightarrow I \longrightarrow 0,\]
 	is the minimal graded free resolution of $I$ and thus the Betti table of $I$ is the following one
 	$$\begin{matrix}
 	& 0 & 1 & 2\\
 	\text{total:}
 	& 8 & 11 & 4\\
 	3: & 8 & 11 & 4
 	\end{matrix}$$
 	Consider the order given in (\ref{eq:orderingG(I)G(P)G(Q)}), \emph{i.e.}, 
 	$$
 	x_2x_4x_6\succ x_2x_4x_7\succ x_2x_4x_8\succ x_2x_5x_7\succ x_2x_5x_8\succ x_2x_6x_8\succ x_1x_6x_8\succ x_1x_5x_8.
 	$$
 	Again, by \textit{Macaulay2} \cite{GDS}, one can verify that $I$ has linear quotients with respect to this order. Firstly, denote by $u_i$ the $i$--th monomial generator of $I$ which appears in the given ordering $\succ$ and set $I_{k-1}= (u_1, \ldots, u_{k-1})$. One has 
 	\[
 	\begin{tabular}{|l|l|}
 	\hline
 	\rule[-3mm]{0mm}{0.8cm}
 	$I_1:u_2=I_1:(x_2x_4x_7)= (x_6)$ &  $I_5:u_6 = I_5: (x_2x_6x_8)= (x_4, x_5)$ \\
 	\hline
 	\rule[-3mm]{0mm}{0.8cm}
 	$I_2:u_3 = I_2: (x_2x_4x_8)=(x_6, x_7)$  & $I_6:u_{7} = I_6: (x_1x_6x_8)= (x_2)$ \\
 	\hline
 	\rule[-3mm]{0mm}{0.8cm}
 	$I_3:u_4 = I_3: (x_2x_5x_7)= (x_4)$ & $I_{7}:u_{8} = I_{7}: (x_1x_5x_8)= (x_2, x_6)$ \\
 	\hline
 	\rule[-3mm]{0mm}{0.8cm}
 	$I_4:u_5 = I_4: (x_2x_5x_8)= (x_4, x_7)$    \\ 
 	\cline{1-1}
 	\end{tabular}
 	\]
 	\vspace{0,2cm}

 	In the next table we collect the relevant sets $\set_I(w)$, with $w\in G(I)=\mathcal{L}_t(u,v)$:
 	
 	\begin{center}
 		$\begin{array}{ccccccccc}
 		\bottomrule[1.05pt]
 		\rowcolor{black!20}&x_2x_4x_6&x_2x_4x_7&x_2x_4x_8&x_2x_5x_7&x_2x_5x_8&x_2x_6x_8&x_1x_6x_8&x_1x_5x_8\\
 		\toprule[1.05pt]\set&\emptyset&\{x_6\}&\{x_6,x_7\}&\{x_4\}&\{x_4,x_7\}&\{x_4,x_5\}&\{x_2\}&\{x_2,x_6\}\\
 		\toprule[1.05pt]
 		\end{array}$
 	\end{center}
 	
 	Let us verify that $\beta_1(I)=11$ by Theorem \ref{thm:Bettiformula}(b). We have
 	\begin{align*}
 	\beta_1(I)&=\sum_{\substack{w\in\mathcal{L}_t(u,v)\\ \min(w)\ge2}}\binom{\max(w)-6}{1}+\sum_{\substack{j=0,1\\w\in\mathcal{L}_t^f(u/x_1)}}\binom{1}{j}\binom{6-\min(w)}{1-j}\\
 	&=1\binom{0}{1}+2\binom{1}{1}+3\binom{2}{1}+\binom{1}{0}\binom{0}{1}+\binom{1}{1}\binom{0}{0}\\
 	&\phantom{=}+\binom{1}{0}\binom{1}{1}+\binom{1}{1}\binom{1}{1}\\
 	&=2+3\cdot2+1+1+1=11.
 	\end{align*}
 	Finally, we verify that $\beta_2(I)=4$.
 	We have
 	\begin{align*}
 	\beta_2(I) &=\sum_{\substack{w\in\mathcal{L}_t(u,v)\\ \min(w)\ge2}}\binom{\max(w)-6}{2}+\sum_{\substack{j=0,1,2\\w\in\mathcal{L}_t^f(u/x_1)}}\binom{1}{j}\binom{6-\min(w)}{2-j}\\
 	&=1\binom{0}{2}+2\binom{1}{2}+3\binom{2}{2}+\binom{1}{0}\binom{0}{2}+\binom{1}{1}\binom{0}{1}\\
 	&\phantom{=}+\binom{1}{2}\binom{0}{0}+\binom{1}{0}\binom{1}{2}+\binom{1}{1}\binom{1}{1}+\binom{1}{2}\binom{1}{0}\\
 	&=3+1=4.
 	\end{align*}
 \end{Example}
 
 \section{Conclusion and Perspectives}
 
 In this article we have investigated $t$--spread lexsegment ideals with linear resolution. In particular, we have focused on incompletely $t$--spread lexsegment ideals. We do not know yet if the corresponding result of Theorem \ref{thm:App2} holds true for completely $t$--spread lexsegment ideals as in the case $t=0$ \cite[Theorem 1.2]{EOS}. Thus, we are led to the next question.
 \begin{Question}
 	Let $I$ be a completely $t$--spread lexsegment ideal with a linear resolution. Is it true that $I$ has linear quotients?
 \end{Question}
 In \cite{CFI1}, in order to study the Cohen--Macaulay property for $t$--spread lexsegment ideals, some formulas for the depth and the Krull dimension have been stated. To the best of our knowledge, it seems that closed formulas for the projective dimension and the regularity of $t$--spread lexsegment ideals are not known also in the $0$--spread case. Some partial results are given in \cite{EOT2010}. As we have underlined in the article, there is always a natural Betti splitting $I=P+Q$ of a $t$--spread lexsegment ideal $I$. Thus, using the following formulas \cite[Corollary 2.2]{FHT2009},
 \begin{align*}
 \pd(I)&=\max\{\pd(P),\pd(Q),\pd(P\cap Q)+1\},\\
 \reg(I)&=\max\{\reg(P),\reg(Q),\reg(P\cap Q)-1\},
 \end{align*}
 one could try to determine closed formulas for such invariants.\\\\
 \emph{Acknowledgement}. We thank the referee for his/her helpful suggestions that allowed us to improve the quality of the paper.

 \end{document}